\documentclass[12pt]{article}
\usepackage{amsmath, amsthm, amssymb, fullpage}

\newtheorem{theorem}{Theorem}[section]
\numberwithin{equation}{theorem}
\newtheorem{lemma}[theorem]{Lemma}

\newtheorem{cor}[theorem]{Corollary}

\newtheorem{question}[theorem]{Question}
\newtheorem{exercise}[theorem]{Exercise}
\theoremstyle{definition}
\newtheorem{defn}[theorem]{Definition}

\newtheorem{remark}[theorem]{Remark}
\newtheorem{hypothesis}[theorem]{Hypothesis}
\newtheorem{notation}[theorem]{Notation}

\newcommand{\del}{\partial}

\newcommand{\CC}{\mathbb{C}}

\newcommand{\QQ}{\mathbb{Q}}
\newcommand{\RR}{\mathbb{R}}
\newcommand{\ZZ}{\mathbb{Z}}

\DeclareMathOperator{\epi}{epi}
\DeclareMathOperator{\Frac}{Frac}

\DeclareMathOperator{\inte}{int}

\DeclareMathOperator{\spect}{spect}

\begin{document}

\title{Detecting integral polyhedral functions}
\author{Kiran S. Kedlaya and Philip Tynan}
\date{November 19, 2008}

\maketitle

\begin{abstract}
We study the class of real-valued functions on convex subsets of
$\RR^n$ which are computed by the maximum of finitely many
affine functionals with integer slopes. We prove several results
to the effect that this property of a function can be detected by
sampling on small subsets of the domain. In so doing, we recover
in a unified way
some prior results of the first author (some joint with Liang
Xiao). We also prove that a function on $\RR^2$ is a tropical polynomial
if and only if its restriction to each translate of a generic tropical
line is a tropical polynomial.
\end{abstract}

\section*{Introduction}

One of the most fundamental classes of real-valued functions on $\RR^n$
is the class of \emph{convex polyhedral functions}, 
i.e., those functions computed as the maximum
of a finite number of affine functionals
\[
\lambda(x_1,\dots,x_n) = a_1 x_1 + \cdots + a_n x_n + b.
\]
For one, this class is fundamental in the theory
of linear programming; on the other hand, it also figures prominently
in algebraic geometry via the study of tropical polynomial functions.

The purpose of this paper is to study some subclasses of convex polyhedral
functions for which we impose some integrality conditions.
The two classes we focus on are the \emph{transintegral}
polyhedral functions, for which the coefficients $a_1,\dots,a_n$
in the affine functionals $\lambda$ must be integers,
and the \emph{integral} polyhedral functions, for which both
the coefficients $a_1, \dots, a_n$ and the constant
term $b$ must be integers. 

What we prove are a number of results of the following form: a function 
on a suitable convex subset of $\RR^n$ is
(trans)integral polyhedral if and only if the same is true of
its restrictions to some small subsets of the domain (usually certain
straight lines). In so doing, we recover in a unified way 
two earlier results along these lines, one 
(Theorem~\ref{T:integral2}) from the solo paper
\cite{kedlaya-part3} by the first author, the other 
(Theorem~\ref{T:kedlaya-xiao}) from the joint
paper \cite{kedlaya-xiao} by the first author and Liang Xiao. 
(Those results were introduced to study $p$-adic differential equations;
we include a brief description of that application.)
We also
obtain a theorem (Theorem~\ref{T:tropical})
that asserts that a tropical Laurent polynomial may
be identified from its restrictions to the translates of a generic
tropical line.

\section{Convex sets and functions}

\begin{notation}
Throughout this paper, let $e_1,\dots,e_n$ denote the standard basis
vectors of $\RR^n$.
\end{notation}

\begin{notation}
For all definitions in this section and the next, fix a subfield $F$ of $\RR$.
When one of these definitions is referenced with $F$ omitted, one should
take $F = \RR$; the only other case we will be interested is $F = \QQ$.
\end{notation}

\begin{defn} \label{D:convex set}
Let $S$ be a subset of $F^n$ for some nonnegative integer $n$.
We say $S$ is \emph{$F$-convex} (resp.\ \emph{$F$-affine}) if 
for any $x,y \in S$ and $t \in [0,1] \cap F$ (resp.\  $t \in F$), 
we have $tx + (1-t)y \in S$. Any intersection of $F$-convex
(resp.\ $F$-affine) sets is again $F$-convex (resp.\ $F$-affine).
\end{defn}

\begin{defn}
For any set $T \subseteq F^n$, the \emph{$F$-convex hull} 
(resp.\ \emph{$F$-affine hull}) of $T$
is the intersection of all $F$-convex (resp.\ $F$-affine)
sets of $F^n$ containing $T$.
It is equal to the set of all points of the form 
$t_1 x_1 + \cdots + t_m x_m$ for some positive integer $m$, some
$x_1,\dots,x_m \in S$, and some $t_1,\dots,t_m \in F \cap [0, +\infty)$
(resp.\ $t_1,\dots,t_m \in F$) with $t_1 + \cdots + t_m = 1$
\cite[Corollary~1.4.1, Theorem~2.3]{rock}.
\end{defn}

\begin{defn}
For $S \subseteq \RR^n$ convex, the \emph{(relative) interior} of $S$,
denoted $\inte(S)$, is defined to be the topological interior of
$S$ relative to its affine hull. This is nonempty if $S$ is nonempty
\cite[Theorem~6.2]{rock}. By the \emph{dimension} of $S$, denoted $\dim(S)$,
we will mean the dimension of its affine hull.
\end{defn}

\begin{defn} \label{D:convex fn}
Let $S \subseteq F^n$ be an $F$-convex set.
A function $f: S \to \RR$ is \emph{$F$-convex} if for any $x,y \in S$ and 
$t \in [0,1] \cap F$, we have the Jensen inequality 
\begin{equation} \label{eq:Jensen ineq}
tf(x) + (1-t)f(y) \geq f(tx + (1-t)y).
\end{equation}
This implies that for any 
$x_1,\dots,x_m \in S$ and $t_1,\dots,t_m \in [0,1] \cap F$
with $t_1 + \cdots + t_m = 1$,
we have
\begin{equation} \label{eq:Jensen ineq2}
\sum_{i=1}^m t_i f(x_i) \geq f \left( \sum_{i=1}^m t_i x_i \right).
\end{equation}
If $f$ is convex, then $f$ is continuous on $\inte(S)$
\cite[Theorem~10.1]{rock}; we will prove a stronger result later
(Theorem~\ref{T:extend convex}).
\end{defn}

\begin{remark}
Note that $f$ is convex if and only if the \emph{epigraph} of $f$, defined by
\[
\epi(f) = \{(x,y) \in S \times \RR: y \geq f(x)\},
\]
is convex. Using this convention, we may extend the definition
of convexity to functions with range $\RR \cup \{\pm \infty\}$
(but only when explicitly specified).
In particular, one can canonically extend any convex function
on $S$ to a convex function on $\RR^n$ taking the value $+\infty$
everywhere on $\RR^n \setminus S$; this is the convention used in
\cite{rock}.
\end{remark}

\section{Directional derivatives}

\begin{defn}
Let $S \subseteq F^n$ be an $F$-convex subset.
Pick $x \in S$ and 
$z \in F^n$ such that $x + tz \in S$ for some $t \in (0, +\infty) \cap F$.
For $f: S \to \RR$ a function whose restriction to
$\{x + tz: t \in [0,\epsilon] \cap F\}$ is $F$-convex for some $\epsilon > 0$,
define $f'(x, z)$ to be the directional derivative of $f$ at $x$
in the direction of $z$, i.e.,
\[
f'(x, z) = \lim_{t \to 0^+} \frac{f(x + tz) - f(x)}{t}.
\]
Note that this is a limit 
taken over a decreasing sequence; for it to exist in all 
cases, we must allow it to take the value $-\infty$.
\end{defn}

\begin{lemma}\label{L:dir deriv}
Let $S \subseteq F^n$ be an $F$-convex subset.
Let $U \subseteq S$ be an $F$-convex subset.
Suppose $z \in F^n$ is such that for each $x \in U$, there exists
$t \in (0, +\infty) \cap F$ such that $x+tz \in S$. 
Let $f: S \to \RR$ be a function satisfying the following conditions.
\begin{enumerate}
\item[(a)]
The restriction of $f$ to $U$ is affine.
\item[(b)]
For each $x \in U$, there exists $\epsilon > 0$ such that
the restriction of $f$ to $\{x+tz: t \in [0, \epsilon] \cap F\}$
is $F$-convex.
\item[(c)]
For each line segment $L$ with endpoints in $U$,
there exists $\epsilon > 0$ such that
for each $t \in [0, \epsilon] \cap F$, the restriction of
$f$ to $(L \cap F^n) + tz$ is $F$-convex.
\end{enumerate}
Then the function $x \mapsto f'(x,z)$ is $F$-convex on $U$.
\end{lemma}
\begin{proof}
Take $x_1, x_2 \in U$.
We assume first that $f'(x_1, z), f'(x_2, z) > -\infty$.
Pick $t \in [0,1] \cap F$ and put $x_3 = tx_1 + (1-t)x_2$.
Choose $\epsilon > 0$ satisfying (b) for each of $x = x_1, x_2, x_3$, and satisfying
(c) for $L$ the line segment from $x_1$ to $x_2$.
Pick $u \in (0, \epsilon] \cap F$ such that
\[
\frac{f(x_i+uz) - f(x_i)}{u} \leq f'(x_i,z) + \delta \qquad (i =1,2).
\]
Then 
\begin{align*}
t f'(x_1,z) + (1-t) f'(x_2,z) & \geq
t \frac{f(x_1 + uz) - f(x_1)}{u} + (1-t) 
\frac{f(x_2 + uz) - f(x_2)}{u} - \delta \\
&= \frac{t f(x_1 + uz) + (1-t)f(x_2 + uz) - f(x_3)}{u} - \delta \\
&\geq \frac{f(x_3 + uz) - f(x_3)}{u} - \delta \\
&\geq f'(x_3,z) - \delta.
\end{align*}
Since $\delta$ was arbitrary, this proves the claim when both 
$f'(x_1, z)$ and $f'(x_2, z)$ are not $-\infty$.  
If one of them is $-\infty$, the same argument would imply that 
$f'(x_3, z) = -\infty$; this completes the proof.
\end{proof}

\section{Affine functionals}

\begin{notation}
For all definitions in this section and the next, fix a subgroup $G$ of $\RR$.
\end{notation}

\begin{defn}
An \emph{affine functional} is a map $\lambda: \RR^n \to \RR$ of the form
$\lambda(x_1,\dots,x_n) = a_1 x_1 + \cdots + a_n x_n + b$ for some
$a_1, \dots, a_n, b \in \RR$.
The \emph{slope} of $\lambda$ is the linear functional
$\mu$ defined by $\mu(x_1,\dots,x_n) = a_1 x_1 + \cdots + a_n x_n$.
We say $\lambda$ is \emph{$G$-integral} if $a_1,\dots,a_n \in \ZZ$
and $b \in G$.
We use \emph{integral} and \emph{transintegral} as synonyms for
$\ZZ$-integral and $\RR$-integral, respectively.
\end{defn}

We can characterize convexity using affine functionals as follows.
\begin{lemma} \label{L:bourbaki}
Let $S \subseteq \RR^n$ be an open convex subset.
Consider the following conditions on a function $f: S \to \RR$.
\begin{enumerate}
\item[(a)]
$f$ is convex.
\item[(b)]
For each $x \in S$, there exists 
an affine functional $\lambda: \RR^n \to \RR$ such that
\begin{align*}
f(x) &= \lambda(x) \\
f(y) &\geq \lambda(y) \qquad (y \in S).
\end{align*}
\item[(c)]
For each $x \in S$, there exist
an affine functional $\lambda$ and
a quantity $\epsilon > 0$
such that 
\begin{align*}
f(x) &= \lambda(x) \\
f(y) &\geq \lambda(y) \qquad (y \in S, |x-y| < \epsilon).
\end{align*}
\end{enumerate}
Then (a) and (b) are equivalent. Moreover,
if $f$ is assumed to be upper semicontinuous,
then (a) and (b) are also equivalent to (c).
\end{lemma}
\begin{proof}
If (a) holds, then (b) holds because $(x, f(x)) \notin \inte(\epi(f))$,
so by \cite[Corollary~11.6.2]{rock} there is a linear function on
$\RR^{n+1}$ which is zero at $(x, f(x))$ and nonnegative on $\epi(f)$.
If (b) holds, then (a) holds because $\epi(f)$ is the intersection of
the epigraphs of some affine functionals, and so is convex.

Clearly (b) implies (c). Conversely, suppose
$f$ is upper semicontinuous function and (c) holds.
As noted in \cite[Theorem~3]{young}, the convexity of $f$ in this case
follows from the analogous statement with $n=1$, which in turn follows
at once from 
\cite[\S 1.4, Proposition~9, Corollaire~1]{bourbaki}.
\end{proof}
\begin{cor} \label{C:minimal ext}
Let $S \subseteq \RR^n$ be an open convex subset,
let $f: S \to \RR$ be a convex function,
and let $T$ be the set of affine functionals $\lambda: \RR^n \to \RR$
such that $f(x) \geq \lambda(x)$ for all $x \in S$. Then for all $x \in S$,
\begin{equation} \label{eq:extend}
f(x) = \sup\{\lambda(x): \lambda \in T\}.
\end{equation}
\end{cor}
\begin{proof}
The inequality $f(x) \geq \sup\{\lambda(x): \lambda \in T\}$ is evident
from the definition of $T$. For the reverse inequality, note that
for each $x \in S$, 
by Lemma~\ref{L:bourbaki}, there exists  $\lambda \in T$ with
$f(x) = \lambda(x)$.
\end{proof}
\begin{cor}
With notation as in Corollary~\ref{C:minimal ext},
the formula \eqref{eq:extend} defines a lower semicontinuous
extension of $f$ to a convex function $\tilde{f}: \RR^n \to \RR \cup \{+\infty\}$.
\end{cor}

\begin{defn} \label{D:domain of aff}
For $S \subseteq \RR^n$ a convex subset and
$f: S \to \RR$ a convex function, a \emph{domain of affinity}
for $f$
is a subset $U \subseteq S$ with nonempty interior relative to 
the affine hull of $S$,
on which $f$ agrees with
an affine functional $\lambda$. 
We call $\lambda$ an \emph{ambient functional} for $U$;
it is uniquely determined if $S$ has affine hull $\RR^n$.
\end{defn}

\begin{lemma} \label{L:domain of aff}
Let $S \subseteq \RR^n$ be a convex subset, let
$f: S \to \RR$ be a convex function, and let
$\lambda$ be an ambient
functional for some domain of affinity $U$ for $f$.
Then $f(x) \geq \lambda(x)$ for all $x \in S$.
\end{lemma}
\begin{proof}
Pick $y$ in the topological interior of $U$ relative to the affine hull
of $S$; then for $t \in [0,1]$ sufficiently close to 0,
$z = tx + (1-t)y$ will also belong to $U$.
The claim now follows from \eqref{eq:Jensen ineq}.
\end{proof}
\begin{cor} \label{C:domain convex}
With notation as in Lemma~\ref{L:domain of aff}, the maximal
domain of affinity $U$ with ambient functional $\lambda$ is a convex set.
\end{cor}
\begin{proof}
For $x,y \in U$ and $t \in [0,1]$, for $z = tx + (1-t)y$, we have
\[
f(z) \geq \lambda(z) = t\lambda(x) + (1-t)\lambda(y)
= tf(x) + (1-t)f(y) \geq f(z).
\]
Hence $z \in U$.
\end{proof}

We  build domains of affinity using the following argument.
\begin{lemma} \label{L:build domain}
Let $T \subseteq \RR^n$ be a subset with convex hull $U$,
and pick $z \in \inte(U)$.
Let
$f: U \to \RR$ be a convex function, let $\lambda: \RR^n \to \RR$
be an affine functional, 
and suppose that $f(x) = \lambda(x)$ for all $x \in T \cup \{z\}$. Then
$f(x) = \lambda(x)$ for all $x \in U$.
\end{lemma}
\begin{proof}
Pick any $x \in U$.
On one hand, by \eqref{eq:Jensen ineq2}, $f(x) \leq \lambda(x)$.
On the other hand, since $z \in \inte(U)$,
there exist $y \in U$ and $t \in (0,1]$ such that
$tx + (1-t) y = z$ (namely, pick $y$ on the ray from $z$ away from $x$). 
Since $y$ is in the convex hull of $T$,
we can find $y_1,\dots,y_m \in T$ and $t_1,\dots,t_m \in [0,1]$
with $t_1 + \cdots + t_m = 1$, such that
$t_1 y_1 + \cdots + t_m y_m = y$.
Then
\[
tx + (1-t)t_1 y_1 + \cdots + (1-t)t_m y_m = z,
\]
so \eqref{eq:Jensen ineq2} implies $f(x) \geq \lambda(x)$.
\end{proof}

\section{Polyhedral sets and functions}

\begin{defn}
A \emph{($G$-rational) closed halfspace} is a subset of $\RR^n$ of the form
\[
S = \{x \in \RR^n: \lambda(x) \geq 0\},
\]
for  $\lambda$ a ($G$-integral) affine functional with nonzero slope.
A subset $S \subseteq \RR^n$ is \emph{($G$-rational) polyhedral}
if it is the intersection of some finite number of 
($G$-rational) closed halfspaces. (This number may be zero, in which
case $S = \RR^n$.) A polyhedral subset of $\RR^n$ is closed and convex,
but not necessarily bounded.
We use \emph{rational} and \emph{transrational} as synonyms for
$\ZZ$-rational (or $\QQ$-rational, as these are equivalent) 
and $\RR$-rational, respectively.
\end{defn}

\begin{lemma} \label{L:minimal ext1}
Let $S \subseteq \RR^n$ be a polyhedral subset.
Then any convex function $f: S \to \RR$ is upper semicontinuous. Consequently,
$f$ is continuous if and only if it is lower
semicontinuous.
\end{lemma}
\begin{proof}
See \cite[Theorem~10.2]{rock}.
\end{proof}

\begin{defn}
Let $S \subseteq \RR^n$ be a polyhedral subset. Choose
affine functionals $\lambda_1,\dots, \lambda_m: \RR^n \to \RR$
such that 
\[
S = \{x \in \RR^n: \lambda_i(x) \geq 0 \qquad (i = 1,\dots,m)\}.
\]
A \emph{facet} of $S$ is a nonempty subset of $S$ of the form
\[
B = \{x \in S: \lambda_i(x) = 0 \qquad (i \in I)\}
\]
for some subset $I$ of $\{1,\dots,m\}$; this definition does not depend
on the choice of $\lambda_1,\dots,\lambda_m$. A facet is \emph{proper}
if it is not equal to $S$. Note that the union of the proper facets
of $S$ equals $S \setminus \inte(S)$. An element $x \in S$ forming a facet
by itself is called a \emph{vertex} of $S$.
\end{defn}

\begin{defn} \label{D:polyhedral}
For $S \subseteq \RR^n$, a function $f: S \to \RR$
is \emph{($G$-integral) polyhedral} if it has the form
\begin{equation} \label{eq:polyhedral}
f(x) = \max\{\lambda_1(x), \dots, \lambda_m(x)\} \qquad (x \in S)
\end{equation}
for some ($G$-integral) affine functionals $\lambda_1,\dots,\lambda_m: \RR^n 
\to \RR$. This implies that $f$ is convex and continuous.
\end{defn}

\begin{lemma} \label{L:poly by domains}
Let $S \subseteq \RR^n$ be a polyhedral subset. Let
$f: S \to \RR$ be a function. Then the following conditions are equivalent.
\begin{enumerate}
\item[(a)]
$f$ is polyhedral.
\item[(b)]
$f$ is convex and $S$ is covered by finitely many
domains of affinity for $f$.
\item[(c)]
$f$ is convex and $S$ is covered by 
domains of affinity for $f$ whose ambient functionals have only
finitely many different slopes.
\end{enumerate}
\end{lemma}
\begin{proof}
If $f(x) = \sup_i\{\lambda_i(x)\}$ with each $\lambda_i$ affine,
then it is clear that $f$ is convex and that the domains of affinity
for $f$ corresponding to the $\lambda_i(x)$ cover $S$. 
Hence (a) implies (b). Conversely,
if $f$ is convex and $S$ is covered by domains of affinity for $f$ with
ambient functionals $\lambda_1,\dots,\lambda_m$, then
$f(x) = \sup_i \{\lambda_i(x)\}$ by Lemma~\ref{L:domain of aff}.
Hence (b) implies (a).

It is trivial that (b) implies (c). The reverse implication holds
by Lemma~\ref{L:domain of aff}, which implies that no two distinct
domains of affinity for $f$ can have ambient functionals with the same slope.
\end{proof}

\begin{defn}
Let $S \subseteq \RR^n$ be a convex subset.
A function $f: S \to \RR$ is \emph{locally ($G$-integral) polyhedral} if
for each $x \in S$, there exists a neighborhood $U$ of $x$ in $\RR^n$
such that the restriction of $f$ to $U \cap S$ is
($G$-integral) polyhedral. If $f$ is locally polyhedral, then $f$ is
evidently continuous; by replacing $\RR^n$ with the affine hull of $S$
and applying Lemma~\ref{L:bourbaki} on $\inte(S)$, we see that $f$ is also
convex.
\end{defn}

\begin{lemma} \label{L:locally to poly}
Let $S \subseteq \RR^n$ be a compact convex subset. Then
$f: S \to \RR$ is locally polyhedral if and only if $f$ is polyhedral.
\end{lemma}
\begin{proof}
Suppose $f$ is locally polyhedral. Then each $x \in S$ admits a neighborhood
$U$ in $\RR^n$ such that the restriction of $f$ to $U \cap S$ is
polyhedral. Since $S$ is compact, it is covered by
finitely many of these neighborhoods. Hence $S$ is covered by finitely
many domains of affinity for $f$; by 
Lemma~\ref{L:poly by domains},
$f$ is polyhedral.
\end{proof}

\section{Extending convex functions}

It is useful to note that the notions of $F$-convexity and ordinary
convexity are essentially equivalent;
this refines the usual assertion that
a convex function on an open convex set 
is continuous \cite[Theorem~10.1]{rock}.

\begin{theorem} \label{T:extend convex}
Let $F$ be a subfield of $\RR$.
Let $S \subseteq \RR^n$ be an open convex subset. Let
$f: S \cap F^n \to \RR$ be a 
function such that for each $x \in F^n$ and each 
$i \in \{1,\dots,n\}$, the restriction of $f$ to 
$S \cap F^n \cap (x + \RR e_i)$ is $F$-convex.
\begin{enumerate}
\item[(a)]
The function $f$ extends uniquely to a continuous function $\tilde{f}$ on $S$
such that for each $x \in \RR^n$ and each 
$i \in \{1,\dots,n\}$, the restriction of $\tilde{f}$ to 
$S \cap (x + \RR e_i)$ is convex.
\item[(b)]
If $f$ is $F$-convex, then $\tilde{f}$ is convex.
\item[(c)]
For any closed intervals $I_1, \dots, I_n$ such that
$I_1 \times \cdots \times I_n \subset S$, the directional derivatives
$f'(x, \pm e_i)$ for $x \in I_1 \times \cdots \times I_n$ and $i \in 
\{1,\dots,n\}$ are bounded.
\end{enumerate}
\end{theorem}
\begin{proof}
We proceed by induction on $n$, with trivial base case $n=0$.
Pick $x = (x_1,\dots,x_n) \in S$ (not necessarily in $S \cap F^n$).
For $i=1,\dots,n$, pick closed intervals $I_{i,0}, I_{i,1}$ with
endpoints in $F$, such that $x_i$ is contained in the interior of $I_{i,0}$, 
$I_{i,0}$ is contained in the
interior of $I_{i,1}$, and $I_{1,1} \times \cdots \times I_{n,1} \subset S$.

Put $B_j = I_{1,j} \times \cdots \times I_{n,j}$ for $j=0,1$.
By the induction hypothesis, $f$ extends uniquely to a continuous
function on each face of each box $B_j$. In particular, $f$
achieves maximum and minimum values on the boundary $\del(B_j)$ of each
$B_j$.

We now argue that $f$ is Lipschitz continuous on $B_0$.
Pick any $y,z \in B_0 \cap F^n$ whose difference is a nonzero
multiple of some $e_i$. Then there exist a unique affine function
$g: \RR \to \RR^n$ and values $a_1 < 0 \leq a_0 < b_0 \leq 1 < b_1$,
all contained in $F$, such that
\begin{align*}
g(a_0) &= y \\
g(b_0) &= z \\
\{t \in \RR: g(t)\in B_0\} &= [0,1] \\
\{t \in \RR: g(t)\in B_1\} &= [a_1,b_1].
\end{align*}
We have $g(a_1), g(0), g(y), g(z), g(1), g(b_1) \in F^n$,
so the $F$-convexity of the restriction of $f$ to the image of $g$ implies
\[
\frac{g(0) - g(a_1)}{- a_1}
\leq \frac{g(z) - g(y)}{b_0 - a_0}
\leq \frac{g(b_1) - g(1)}{b_1 - 1}.
\]
In this expression, the quantities $g(a_1), g(0), g(1), g(b_1)$ are bounded
because $f$ is bounded on $\del(B_0) \cup \del(B_1)$.
The quantities $- a_1$ and $b_1 - 1$ depend only on $i$, so they are 
bounded above and bounded below away from 0.

It now follows that for some $c>0$,
for all $y = (y_1,\dots,y_n), z = (z_1,\dots,z_n)
\in B_0 \cap F^n$, we have 
\[
|f(y) - f(z)| \leq c (|y_1 - z_1| + \cdots + |y_n - z_n|).
\]
This implies that $f$ is Lipschitz continuous on $B_0$,
so $f$ extends uniquely
to a continuous function on $B_0$. 

Since $B_0$ was chosen to contain
an arbitrary $x \in S$ in its interior, we conclude that $f$ extends
uniquely to a continuous function on all of $S$, proving (a).
We deduce (b) as an immediate corollary
because the terms of \eqref{eq:Jensen ineq} vary continuously with the 
arguments. We deduce (c) from the Lipschitz property established above.
\end{proof}

\section{Detecting locally polyhedral functions}

We now give a criterion for local $F$-integral polyhedrality.

\begin{theorem} \label{T:master}
Let $F$ be a subfield of $\RR$.
Let $S \subseteq \RR^n$ be an open convex subset.
Let $f: S \cap F^n \to \RR$ be a continuous function. 
Suppose that the restriction 
of $f$ to every $F$-rational line segment contained in $S$ and
parallel to one of the coordinate axes
is $F$-integral polyhedral. 
Then the continuous extension of $f$ to $S$ given by Theorem~\ref{T:extend
convex} is locally $F$-integral polyhedral.
\end{theorem}
The proof will be by induction on $n$, with trivial base case $n=1$. 
In order to break the proof up
into a sequence of lemmas, we must assert a hypothesis that will be
available during the induction step.

\begin{hypothesis} \label{H:master}
Fix a value of $n$,
and set notation as in Theorem~\ref{T:master}.
By Theorem~\ref{T:extend convex},
$f$ extends uniquely
to a continuous function whose restriction to each line parallel to
a coordinate axis is convex; we will also call this extended function $f$.
Also assume that 
the conclusion of Theorem~\ref{T:master} holds in all cases where
$n$ is replaced by any smaller positive integer.
\end{hypothesis}

\begin{lemma}\label{L:integral derivs}
Let $I \subseteq \RR$ be an open subinterval. Let 
$f: I \to \RR$ be a convex function such that $f'(x, 1) \in \ZZ$ for all
$x \in \QQ$. Then $f$ is locally transintegral polyhedral.
\end{lemma}
\begin{proof}
For any two $x_1,x_2 \in I$ for which $f'(x_1,1) = f'(x_2,1)$ and any
$x_3 \in (x_1, x_2)$, we have
\begin{align*}
f'(x_1,1) &\leq \frac{f(x_3)-f(x_1)}{x_3 - x_1} \\
&\leq \frac{f(x_2) - f(x_3)}{x_2 - x_3} \\
&\leq f'(x_2, 1) = f'(x_1,1).
\end{align*}
Hence $f$ is affine with integral slope on $[x_1, x_2]$.

Let $J = [a,b]$ be any closed subinterval of $\inte(I)$.
The values $f'(x,1)$ for $x \in I$ are integral and lie in the finite
interval $[f'(a,1), f'(b,1)]$, so they are restricted to a finite set. 
Since $f'(x,1)$ is nondecreasing,
the set of $x$ for which $f'(x,1)$ takes any particular value is
connected. On such a set, by the previous paragraph
$f$ is affine with integral slope. The closures of these sets cover
$J$, so $f$ is transintegral polyhedral on $J$. This proves that
$f$ is locally transintegral polyhedral on $I$.
\end{proof}

\begin{lemma} \label{L:all trans}
Under Hypothesis~\ref{H:master},
for any $x \in S$ and any $i \in \{1,\dots,n\}$,
the restriction of $f$ to $x + \RR e_i$ is locally
transintegral polyhedral.
\end{lemma}
\begin{proof}
It suffices to treat the case $i=n$.
Write $x = (x_1,\dots,x_n)$. 
By Lemma~\ref{L:integral derivs}, it suffices to check that if 
$x_n \in F$, then $f'(x,e_n) \in \ZZ$.

Let $H$ denote the hyperplane spanned by $e_1,\dots,e_{n-1}$.
Since $x_n \in F$, the hyperplane $x + H$ is $F$-rational.
By Hypothesis~\ref{H:master},
the restriction of $f$ to
$x + H$ is locally $F$-integral polyhedral.
Consequently, there exists an $F$-rational polyhedral
subset $U$ of $x+H$ such that $x \in \inte(U)$
and $f$ is affine on $U$.
The set $\inte(U) \cap F^n$ is dense in $U$ because $U$ is 
$F$-rational polyhedral.
By Lemma~\ref{L:dir deriv}, the function $f'(y, e_n)$ is 
$F$-convex for $y \in \inte(U) \cap F^n$; on the other hand,
for $y \in \inte(U) \cap F^n$,
we have $f'(y, e_n) \in \ZZ$ by Hypothesis~\ref{H:master}.
Hence $f'(y,e_n)$ is equal to a constant integer value on all of $\inte(U)$.
In particular, $f'(x, e_n) \in \ZZ$, completing the proof.
\end{proof}

\begin{lemma} \label{L:master convex}
Under Hypothesis~\ref{H:master}, $f$ is convex
and each $x \in S$ belongs to a domain of affinity for $f$
whose ambient functional is transintegral.
\end{lemma}
\begin{proof}
By Lemma~\ref{L:all trans}, we may reduce to the case $F = \RR$.
We will use the criterion of Lemma~\ref{L:bourbaki} to prove convexity; 
it suffices
to show that for each $x \in S$, we can find a neighborhood $U$ of $x$ in $S$
and a convex function $g: U \to \RR$ such that $f(x) = g(x)$ and 
$f(y) \geq g(y)$ for $y \in U$. In the process, we will exhibit a domain of
affinity for $f$ containing $x$.

Let $H$ denote the hyperplane spanned by $e_1,\dots,e_{n-1}$.
By Hypothesis~\ref{H:master},
the restriction of $f$ to
$x + H$ is locally transintegral polyhedral. In particular, of the domains
of affinity for $f$ on $S \cap (x + H)$, only finitely many contain $x$.
Call these $V_1,\dots,V_m$, and let
$\lambda_1,\dots,\lambda_m: x + H \to \RR$ be the corresponding 
ambient functionals; then by Lemma~\ref{L:domain of aff},
\[
f(y) \geq \sup \{\lambda_1(y), \dots, \lambda_m(y)\} \qquad (y \in S \cap (x+H))
\]
with equality at $y = x$. 

For $j = 1, \dots, m$, extend $\lambda_j$ to an affine functional
on $\RR^n$ which is constant on $\RR e_n$. 
Pick any $z_j \in \inte(V_j)$, and define a function
$\mu_j: (S \cap (x+H)) + \RR e_n \to \RR$ by setting
\[
\mu_j(y + te_n) = \begin{cases}
\lambda_j(y) + t f'(z_j, e_n) & (t \geq 0) \\
\lambda_j(y) - t f'(z_j, -e_n) & (t \leq 0)
\end{cases}
\qquad (y \in S \cap (x+H)).
\]
This definition does not depend on the choice of $z_j$, by an argument
as in the proof of Lemma~\ref{L:all trans}: 
$f'(z_j,\pm e_n)$ takes integer values but
is convex in $z_j$ by Lemma~\ref{L:dir deriv}, 
so must be constant on $\inte(V_j)$.
Moreover, for $\epsilon > 0$ sufficiently small, we may choose
$z_{j,1},\dots,z_{j,n+1} \in \inte(V_j)$ with
$z_{j,n+1}$ in the interior of the simplex with vertices
$z_{j,1},\dots,z_{j,n}$, such that
for $k=1,\dots,n+1$, $f(z_{j,k} + te_n)$ is affine for $t \in [0, \epsilon]$.
Then the convexity of $f$ on $S \cap (x + H + te_n)$ 
(given by Hypothesis~\ref{H:master})
and Lemma~\ref{L:build domain} imply that for $t \in [0, \epsilon]$, 
the restriction of $f$ to $x + H + te_n$ admits a domain of affinity
with ambient functional $\mu_j$. By Lemma~\ref{L:domain of aff},
$f(y) \geq \mu_j(y)$ for $y \in S \cap (x + H + te_n)$.
By arguing similarly for negative $t$, we conclude that
for some $\epsilon > 0$,
\begin{equation} \label{eq:support convex}
f(y) \geq \sup\{\mu_1(y),\dots,\mu_m(y)\} \qquad
(y \in S \cap ((x+H) + [-\epsilon,\epsilon] e_n)),
\end{equation}
with equality at $y = x$. 

By Hypothesis~\ref{H:master}, the restriction of $f$ to
$S \cap (z_j + \RR e_n)$ is convex, so
$f'(z_j,e_n) \geq -f'(z_j,-e_n)$. 
Hence $\mu_j$ is convex, so \eqref{eq:support convex}
implies that the criterion of Lemma~\ref{L:bourbaki} is satisfied at $x$.
Since $x \in S$ was arbitrary, we deduce that $f$ is convex.
Moreover, for each $j$, 
\[
(y_1,\dots,y_n) \mapsto \lambda_j(y_1,\dots,y_{n-1}) + f'(z_j, e_n) y_n
\]
is a transintegral affine functional which agrees with $f$  both
at $x$ and on an open subset of $S$, so its corresponding domain
of affinity contains $x$.
\end{proof}

We are now ready to prove Theorem~\ref{T:master}.
\begin{proof}[Proof of Theorem~\ref{T:master}]
We induct on $n$ with trivial base case $n=1$. Under the induction
hypothesis, 
$f$ extends continuously to $S$ by Theorem~\ref{T:extend convex},
and this extended function is convex by Lemma~\ref{L:master convex}.

We first check that $f$ is polyhedral on any 
closed box $B$ contained in $S$. By Lemma~\ref{L:master convex},
$B$ is covered by domains of affinity for $f$ whose ambient
functionals are transintegral; by
Theorem~\ref{T:extend convex}, the slopes of these functionals
are bounded. Hence there can be only finitely many slopes;
by Lemma~\ref{L:poly by domains}, $f$ is polyhedral.

It now follows that $f$ is locally polyhedral. To check that each ambient
functional is $F$-integral polyhedral, we simply restrict to
some $F$-rational line segments of positive length parallel to 
$e_1,\dots,e_n$ contained in the corresponding domain
of affinity.
\end{proof}

\section{Detecting polyhedral functions on polyhedra}

One can ask also for criteria for detecting $F$-integral polyhedrality
of a function. To articulate results as strong as possible in this
direction, we assume that the function is already known to be 
locally $F$-integral polyhedral; then by Lemma~\ref{L:poly by domains},
it suffices to check that there are only finitely many
ambient functionals. 

\begin{lemma} \label{L:constant dir}
Let $S \subseteq \RR^n$ be a convex subset.
Let $U \subseteq S$ be a convex subset.
Suppose $z \in \RR^n$ is such that for any $x \in U$,
there exists $t > 0$ for which $x + tz \in S$.
Let $f: S \to \RR$ be a locally polyhedral function whose restriction
to $U$ is affine. Then $f'(x,z)$ is constant for $x \in \inte(U)$.
\end{lemma}
\begin{proof}
The function $f'(x,z)$ is convex on $\inte(U)$ by
Lemma~\ref{L:dir deriv}. On the other hand, since $f$ is locally polyhedral,
$f'(x,z)$ is locally limited to a finite set. Hence $f'(x,z)$ must be
constant on $\inte(U)$. (Compare the proof of Lemma~\ref{L:all trans}.)
\end{proof}

\begin{defn}
An \emph{affine orthant} in $\RR^n$ is a subset of $\RR^n$ of the form
\[
\{x \in \RR^n: \lambda_i(x) \geq 0 \qquad (i=1,\dots,n)\}
\]
for some affine functionals $\lambda_1,\dots,\lambda_n$
whose slopes are linearly independent.
\end{defn}

\begin{lemma} \label{L:bound1}
Let $S \subseteq \RR^n$ be a polyhedral subset with affine hull $\RR^n$,
for some $n \geq 2$. Assume that $S$ is contained in an affine orthant.
Let $f: S \to \RR^n$ be a continuous convex function.
Suppose that the restriction of $f$ to $\inte(S)$ is locally transintegral
polyhedral, and that the restriction of $f$ to each proper facet of $S$
is polyhedral. Then $f$ is transintegral polyhedral.
\end{lemma}
\begin{proof}
Since the restriction of $f$ to each proper facet of $S$ is polyhedral,
we can partition $S \setminus \inte(S)$ into finitely many subsets
$P_1, \dots, P_k$, such that each $P_i$ is the interior of a
polyhedral subset contained in some
proper facet, and $f$ is affine on each $P_i$.

Suppose $S$ is contained in the affine orthant defined by the affine
functionals $\lambda_1,\dots,\lambda_n$ with slopes $\mu_1,\dots,\mu_n$.
Let $z_1,\dots,z_n$ be a basis of $\RR^n$ such that for each 
$i \in \{1,\dots,n\}$, the quantities $\mu_1(z_i), \dots, \mu_n(z_i)$
are all nonzero and not all of the same sign.

Suppose $i\in \{1,\dots,n\}$,
$j \in \{1,\dots,k\}$, $x_1, x_2 \in P_j$ are such that for some $i$,
$f'(x_1,z_i)$ and $f'(x_2, z_i)$ are both defined. Then there is 
an open convex set $U \subseteq P_j$ containing both $x_1$ and $x_2$,
such that $f'(x,z_i)$ is defined for all $x \in U$. By Lemma~\ref{L:constant
dir}, $f'(x_1,z_i) = f'(x_2,z_i)$.

Now pick any $x \in S$ and put $I_i = \{t \in \RR: x + tz_i \in S\}$.
By the choice of $z_i$, the interval $I_i$ is always bounded, so we
can write it as $[a,b]$. By the previous paragraph, $f'(x + az_i, z_i)$
and $f'(x + bz_i, -z_i)$ are both limited to finite sets. By convexity,
this limits $f'(x, z_i)$ to a bounded interval. 

Applying this for
$i=1,\dots,n$, we deduce that for any ambient functional $\lambda$
for $f$ with slope $\mu$, $\mu(z_1), \dots,\mu(z_n)$ are limited to bounded
intervals. Since $\mu$ is integral, this limits $\mu$ to a finite set.
By Lemma~\ref{L:poly by domains}, $f$ is polyhedral, as desired.
\end{proof}

At this point, we can already recover a result of the first author and Liang
Xiao \cite[Theorem~3.2.4]{kedlaya-xiao}, by specializing the following
theorem to the case $F = \RR$.
\begin{theorem} \label{T:kedlaya-xiao}
Let $F$ be a subfield of $\RR$.
Let $S$ be an $F$-rational polyhedral subset.
Let $f: S \cap F^n \to \RR$ be a function whose restriction to
any $F$-rational line is $F$-integral
polyhedral. Then $f$ is $F$-integral polyhedral.
\end{theorem}
\begin{proof}
By Theorem~\ref{T:extend convex},
the restriction of $f$ to $\inte(S)$ extends uniquely to a continuous
convex function $g: \inte(S) \to F^n$.
By Theorem~\ref{T:master},  $g$ is locally $F$-integral polyhedral on 
$\inte(S)$.
Use the formula \eqref{eq:extend} to extend $g$ to a lower
semicontinuous convex function $g: S \to \RR$; then $g$ is continuous
on $S$ by Lemma~\ref{L:minimal ext1}.
For each $F$-rational line $L$ meeting $\inte(S)$, 
$f$ and $g$ agree on $\inte(S) 
\cap F^n \cap L$. In particular, both functions are $F$-integral
polyhedral on $\inte(S) \cap F^n \cap L = \inte(S \cap L) \cap F^n$. 
Since $f$ is also
$F$-integral polyhedral on $S \cap F^n \cap L$ while $g$
is continuous on that same domain, we conclude that $f$ and $g$
agree on $S \cap F^n \cap L$. Since each element of $S \cap F^n$
lies on an $F$-rational line meeting $\inte(S)$, we conclude that
$f$ and $g$ agree on all of $S \cap F^n$.

To prove that $f$ is $F$-integral polyhedral, it now suffices
to prove that $g$ is $F$-integral polyhedral.
It suffices to check this with $S$ replaced by
$S \cap (\RR c_1 e_1 + \cdots + \RR c_n e_n)$ for each
$(c_1,\dots,c_n) \in \{\pm 1\}^n$. In each of these cases, we may
deduce the claim by induction on $\dim(S)$,
using Theorem~\ref{L:bound1}.
\end{proof}

We next prove a much stronger form of Theorem~\ref{T:kedlaya-xiao}.
\begin{lemma} \label{L:bound2}
Assume one of the following sets of hypotheses.
\begin{enumerate}
\item[(a)]
Let $x_1, x_2, x_3 \in \RR^2$ be distinct points
such that the segment $x_2x_3$ is transrational.
Let $\ell_1, \ell_2$ denote the segments $x_1x_2, x_2x_3$.
\item[(b)]
Let $x_1, x_2 \in \RR^2$ be distinct points.
Let $\ell_1, \ell_2$ be parallel closed transrational rays emanating
from $x_1, x_2$.
\end{enumerate}
Let $S$ be the convex hull of $x_1 x_2 \cup \ell_1 \cup \ell_2$.
Suppose that
$f: S \to \RR \cup \{+\infty\}$ is a lower semicontinuous convex function such that 
the restriction of $f$ to $\inte(S)$ (takes finite values and)
is locally transintegral polyhedral,
and the restrictions of $f$ to $x_1 x_2$ and $\ell_1$ (take finite values and)
are polyhedral.
Then $f$ takes finite values and is polyhedral.
\end{lemma}
\begin{proof}
Let $z \in \QQ^2$ be a nonzero vector parallel to $\ell_2$ in the direction
away from $x_2$. We check that 
for $x \in S \setminus \ell_2$,
$f'(x,z)$ is limited to a bounded set independent of $x$.
\begin{enumerate}
\item[(a)]
In this case, we may argue just as in Lemma~\ref{L:bound1}:
there exist $a \leq b \in \RR$ with $x + az \in x_1 x_2$ and $x + bz \in 
\ell_1$,
$f'(x,z)$ is trapped between $f'(x+az,z)$ and $-f'(x+bz,-z)$
by convexity, and each of those is limited to a finite set.
\item[(b)]
As in (a), we see that $f'(x,z)$ is bounded below.
Put
\[
P = \{(a,b) \in \RR^2: 0 \leq a, 0 \leq b, a + b \leq 1\}.
\]
Define the bijection $h: P \setminus \{(1,0)\} \to S$ by
\[
h(a,b) = x_1 + \frac{a}{2-2a-b} z + \frac{b}{2-2a-b} (x_2-x_1).
\]
Then for any affine functional $\lambda$ on $\RR^2$,
$(2-2a-b) \lambda \circ h$ is again an affine functional on $\RR^2$
(although transintegrality may not be preserved).
Consequently, $F = (2-2a-b) f \circ h$ is locally polyhedral on
$\inte(P)$ and polyhedral on the horizontal and vertical proper facets
of $P$. 
Using the formula \eqref{eq:extend}, we may extend $F$ to
a lower semicontinuous convex function $F: P \to \RR \cup \{+\infty\}$.
Since $F$ is bounded on a subset (the union of the horizontal and 
vertical facets) with convex hull $P$, $F$ takes finite values everywhere.
Hence $F$ is continuous by Lemma~\ref{L:minimal ext1}.

For any $c,d \geq 0$ not both zero with $d < c$, we have
\begin{align*}
F(1,0) &= \lim_{s \to 0^+} F(1-sc, sd) \\
&= \lim_{s \to 0^+}
(2sc - sd) f\left( x_1 + \frac{1-sc}{2sc-sd} z + \frac{sd}{2sc-sd} (x_2-x_1) \right).
\end{align*}
If $\lambda$ is an ambient functional for $f$ with slope $\mu$, by
Lemma~\ref{L:domain of aff} we have
\begin{align*}
F(1,0) &\geq \lim_{s \to 0^+} 
(2sc - sd) \lambda\left( x_1 +  \frac{1-sc}{2sc-sd} z + \frac{sd}{2sc-sd} (x_2 - x_1) \right) \\
&= \lim_{s \to 0^+} 
\mu( (1-sc) z + sd (x_2- x_1)) \\
&= \mu(z).
\end{align*}
Hence $\mu(z)$ is bounded above, as then is $f'(x,z)$ for any $x \in S \setminus \ell_2$.
\end{enumerate}
Since $f$ is locally transintegral
polyhedral and $z \in \QQ^2$, 
$f'(x,z)$ is in fact limited to a \emph{finite} set $T$.
For $m \in T$ and $t \in [0,1)$, define
\[
g_m(t) = \inf_u \{f(x_1 + t(x_2-x_1) + uz) - mu\}.
\]
A straightforward calculation shows that $g_m(t)$ is convex,
as follows.
Given $t_1, t_2 \in [0,1)$ and $w \in [0,1]$, put $t_3 = wt_1 + (1-w)t_2$.
For any $\delta > 0$, we may find $u_1, u_2$ with
\[
f(x_1 + t_i(x_2-x_1) + u_i z) - mu_i \leq g_m(t_i) + \delta \qquad (i=1,2);
\]
then for $u_3 = w u_1 + (1-w)u_2$, we have
\begin{align*}
w g_m(t_1) + (1-w) g_m(t_2) &\geq 
w (f(x_1 + t_1(x_2-x_1) + u_1 z) - mu_1) \\
&+
(1-w) (f(x_1 + t_2(x_2-x_1) + u_2 z) - mu_2)  - \delta\\
&\geq f(x_1 + t_3(x_2 - x_1) + u_3 z) - mu_3 - \delta \\
&\geq g_m(t_3) - \delta.
\end{align*}
Since $\delta>0$ was arbitrary, $g_m(t)$ must be convex.

We may extend $g_m(t)$ to a lower semicontinuous convex function
$g_m: [0,1] \to \RR \cup \{+\infty\}$ using \eqref{eq:extend}.
Note that $g_m(t) \leq f(x_1 + t(x_2-x_1)) \leq \max\{f(x_1), f(x_2)\}$
for all $t \in [0,1)$, so $g_m(1) < +\infty$. By Lemma~\ref{L:minimal ext1},
$g_m: [0,1] \to \RR$ is continuous.

For $t \in [0,1)$ and $u$ such that $x_1 + t(x_2 - x_1) + uz \in S$,
we have
\begin{equation} \label{eq:bound}
f(x_1 + t(x_2-x_1) + uz) = \sup_m \{g_m(t) + mu\}.
\end{equation}
The right side of \eqref{eq:bound}
extends to a continuous convex function on all of $S$
with finite values. Since the left side of \eqref{eq:bound} 
is convex and lower 
semicontinuous, it must also take finite values; hence both sides
of \eqref{eq:bound}
are continuous by Lemma~\ref{L:minimal ext1}, and thus must coincide.

The right side of \eqref{eq:bound}, when restricted to $t=1$,
is polyhedral. Hence $f$ is polyhedral on $x_1 x_2$,
so Lemma~\ref{L:bound1} implies that $f$ is polyhedral on $S$.
\end{proof}

\begin{theorem} \label{T:skeleton}
Let $F$ be a subfield of $\RR$.
Let $S \subseteq \RR^n$ be an $F$-rational polyhedral 
subset contained in an affine
orthant.
Let $f: \inte(S) \to \RR$ be a locally $F$-integral polyhedral function.
Suppose that there exist $\ell_1, 
\dots, \ell_m$ with the following properties.
\begin{enumerate}
\item[(a)]
For $i=1,\dots,m$, $\ell_i$ is a convex subset of a line
(not necessarily $F$-rational) and $\inte(\ell_i) \subseteq \inte(S)$.
\item[(b)]
For $i = 1,\dots,m$, $f$ is polyhedral on $\inte(\ell_i)$.
\item[(c)]
For each vertex $Q$ of $S$, there is some $i$ such that 
$Q \in \ell_i$.
\item[(d)]
For each unbounded one-dimensional facet $B$ of $S$, 
there is some $i$ such that $\ell_i$ is a translate of $B$.
\end{enumerate}
Then $f$ is polyhedral.
\end{theorem}
\begin{proof}
We immediately reduce to the case where $\dim(S) = n$.
If $n = 1$, then the $\ell_i$ must cover $S$ and 
so $f$ is automatically polyhedral.
We thus assume $n \geq 2$ hereafter.

Extend $f$ to a lower semicontinuous function $f: S \to \RR \cup \{+\infty\}$
using \eqref{eq:extend}. 
Let $y$ be any vertex of $S$. By hypothesis, one of the $\ell_i$ has endpoint
$y$. Pick $x \in \inte(\ell_i)$; then $x \in \inte(S)$. Then for any 
$z \in \inte(S)$ such that $yz$ is $F$-rational, we may apply
Lemma~\ref{L:bound2} to deduce that $f$ is polyhedral on $yz$.

Let $B$ be a one-dimensional facet of $S$. 
Since $S$ is contained in an affine orthant, $B$ has at least 
one endpoint $y$.
\begin{enumerate}
\item[(a)]
If $B$ is bounded, let $z$ be its other endpoint. 
Choose a two-dimensional $F$-rational plane containing $B$ and passing
through $\inte(S)$. We can find a point $x$ on this plane 
so that $x \in \inte(S)$, and $xy$ and $xz$ are
both $F$-rational. We have from above that $f$ is polyhedral on $xy$ and $xz$;
it is thus polyhedral on $B$ by Lemma~\ref{L:bound2}.
\item[(b)]
If $B$ is unbounded, pick any $x \in \inte(S)$ such that 
$xy$ is $F$-rational. 
By hypothesis,
one of the $\ell_i$ is a translate of $B$. Let $R$ be the ray from $x$
parallel to $B$; then by Lemma~\ref{L:bound2} (applied to $x$ and $\ell_i$),
$f$ is polyhedral on $R$. We have from above that $f$ is polyhedral on
$xy$; thus by Lemma~\ref{L:bound2} again, $f$ is polyhedral on
$B$.
\end{enumerate}
In either case, $f$ is polyhedral on each one-dimensional facet of $S$.
By Lemma~\ref{L:bound1} and induction on dimension, $f$ is polyhedral
on all of $S$, as desired.
\end{proof}

As an immediate corollary, we now obtain a much stronger form of
Theorem~\ref{T:kedlaya-xiao}.
\begin{cor} \label{C:kedlaya-xiao2}
Let $F$ be a subfield of $\RR$.
Let $S$ be an $F$-rational polyhedral subset.
Let $f: \inte(S) \cap F^n \to \RR$ be a function whose restriction to
any $F$-rational line is $F$-integral
polyhedral. Then $f$ is $F$-integral polyhedral.
\end{cor}
The difference between this result and Theorem~\ref{T:kedlaya-xiao}
is that we do not assume anything about lines contained in the boundary
of $S$.
\begin{proof}
Extend $f$ to a continuous convex function on $\inte(S)$ using
Theorem~\ref{T:extend convex}.
By Theorem~\ref{T:master}, $f$ is locally $F$-integral polyhedral;
by Theorem~\ref{T:skeleton}, $f$ is polyhedral.
\end{proof}

It should also be possible to formulate Theorem~\ref{T:skeleton} without
assuming that $S$ lies in an affine orthant. We leave the following
as an exercise.

\begin{exercise}
Let $F$ be a subfield of $\RR$.
Let $S \subseteq \RR^n$ be an $F$-rational polyhedral subset.
Let $f: \inte(S) \to \RR$ be a locally $F$-integral polyhedral function.
Suppose that there exist $\ell_1, \dots, \ell_m$ 
with the following properties.
\begin{enumerate}
\item[(a)]
For $i = 1,\dots,m$, 
$\ell_i$ is a convex subset of a line (not necessarily $F$-rational)
and $\inte(\ell_i) \subseteq \inte(S)$.
\item[(b)]
For $i = 1,\dots,m$, $f$ is polyhedral on $\inte(\ell_i)$.
\item[(b)]
The convex hull of $\ell_1 \cup \cdots \cup \ell_m$ is equal to $S$.
\end{enumerate}
Then $f$ is polyhedral.
\end{exercise}

\section{Integral polyhedrality from values}

If we restrict attention to integral polyhedral functions, we can do better
than characterizing them 
by their restrictions to one-dimensional rational polyhedra.
We can in fact identify them in terms of their restrictions to
zero-dimensional rational polyhedra, i.e., to rational points,
as long as we appropriately supplement with continuity and 
convexity hypotheses.

\begin{lemma} \label{L:one-dim values}
Let $a \leq b$ be rational numbers. Let $f: [a,b] \to \RR$ be a continuous
convex function such that 
\[
f(x) \in \ZZ + \ZZ x \qquad (x \in (a,b) \cap \QQ).
\]
Then $f$ is integral polyhedral.
\end{lemma}
\begin{proof}
(Compare \cite[Lemmas~2.3.1 and~2.4.1]{kedlaya-part3}.)
We first check that the restriction of $f$ to $(a,b)$ is locally integral 
polyhedral.
Given $x \in (a,b) \cap \QQ$, write $x = r/s$ in lowest terms.
For $N$ any sufficiently large positive integer, we have
\[
\frac{f(x + 1/(sN)) - f(x)}{1/(sN)} 
\in sN (\ZZ + \ZZ x + \ZZ(sN)^{-1}) = \ZZ.
\]
As $N \to \infty$,
this difference quotient runs through a sequence of integers which is nonincreasing
(because $f$ is convex)
and bounded below by $-f'(x,-1)$.(because $x \in (a,b)$).
Thus the quotient stabilizes for $N$ large. By convexity, the function
$f$ must be affine with integral slope in a one-sided neighborhood of $x$.
By Lemma~\ref{L:integral derivs}, $f$ is locally transintegral polyhedral;
by comparing values (e.g., for one irrational $x$ in each domain of affinity), 
we see that the constant terms must be in $\ZZ$. Hence
$f$ is locally integral polyhedral.

We next check that $f$ is integral polyhedral.
Suppose 
that the graph of $f$ has slopes which are unbounded below.
The supporting lines of these slopes intersect the vertical line
$x=a$ in points whose $y$-coordinates form a strictly increasing
sequence within the discrete group $\ZZ + \ZZ a$.
This sequence is bounded above by $f(a)$, contradiction.

That contradiction shows that the slopes of
$f$ are bounded below; similarly, the slopes of $f$ are bounded above.
Hence $f$ has only finitely many slopes 
on $(a,b)$, hence is integral polyhedral.
\end{proof}

We now obtain the following result, which is a slight strengthening of 
\cite[Theorem~2.4.2]{kedlaya-part3}.
\begin{theorem} \label{T:integral2}
Let $S \subset \RR^n$ be a bounded rational polyhedral subset.
Let $f: S \cap \QQ^n \to \RR$ be a $\QQ$-convex function 
whose restriction to each rational line segment is continuous,
such that
\[
f(x_1, \dots, x_n) \in \ZZ + \ZZ x_1 + \cdots + \ZZ x_n
\qquad ((x_1, \dots, x_n) \in S \cap \QQ^n).
\]
Then $f$ is integral polyhedral.
\end{theorem}
\begin{proof}
By Lemma~\ref{L:one-dim values}, 
for each rational line $L$, the restriction of $f$ to $S \cap \QQ^n \cap L$
is integral polyhedral. 
By Theorem~\ref{T:kedlaya-xiao}, 
$f$ is $\QQ$-integral polyhedral.

To check that $f$ is integral polyhedral, we may reduce to the case
$\dim(S) = n$. Extend $f$ to a $\QQ$-integral polyhedral function
on $S$. Let $U$ be a maximal domain of affinity for $f$; since
$f$ is $\QQ$-integral polyhedral, $U$ is rational polyhedral.
Let $\lambda(x) = a_1 x_1 + \cdots + a_n x_n + b$ 
be the ambient functional for $U$; since $f$
is $\QQ$-integral polyhedral,
$a_1,\dots,a_n \in \ZZ$ and $b \in \QQ$.
Write $b = r/s$ in lowest terms.
Pick $(x_1,\dots,x_n) \in U \cap \QQ^n$ such that $s'x_1, \dots, s'x_n 
\in \ZZ$ for some integer $s'$ coprime to $s$ (such points are dense in $U$).
Then on one hand $f(x_1,\dots,x_n) \in (s')^{-1} \ZZ$, while on the other hand
$\lambda(x) \in (s')^{-1} \ZZ + b$. We thus have $b \in s^{-1} \ZZ \cap
(s')^{-1} \ZZ = \ZZ$, so $f$ is integral polyhedral.
\end{proof}

\section{Tropical polynomials}

The subject of \emph{tropical
algebraic geometry} has become quite active lately. 
In that subject, one works not with an ordinary
ring but with the \emph{tropical semiring}, in which the underlying
set is $\RR$, the ``addition'' operation is the maximum, and the
``multiplication''  is ordinary addition. (It is more customary to
take the minimum instead of the maximum, for better correspondence
with valuation theory, but it is more consistent
with the notation in this paper to use the opposite sign convention.)

One is led to ask what a ``tropical polynomial'' is. If we imagine a
polynomial in the variable $x_1, \dots, x_n$
to be a ``sum'' of terms each of which is the ``product''
of a constant with some of the $x_i$ (possibly repeated), we see our
answer at once: a tropical polynomial is merely a transintegral polyhedral
function in which the slopes have nonnegative coefficients. Similarly,
a tropical Laurent polynomial, in which we allow ``dividing'' by 
$x_1,\dots,x_n$ as well, is none other than an arbitrary transintegral
polyhedral function.

It is then reasonable to ask for statements identifying tropical (Laurent)
polynomials from their restrictions to certain ``tropical lines''.
Here is a sample statement.
\begin{theorem} \label{T:tropical}
Let $f: \RR^2 \to \RR$ be a function. Then 
the following are equivalent.
\begin{enumerate}
\item[(a)]
The function $f$ is transintegral polyhedral.
\item[(b)]
The restriction of $f$ to each horizontal and vertical line is 
transintegral polyhedral.
\item[(c)]
The restriction of $f$ to each translate of the set
\begin{gather*}
L = \{(x,y) \in \RR^2: x = 0, y \leq 0 \}
\cup \{(x,y) \in \RR^2: x \leq 0, y = 0\} \\
\cup \, \{(x,y) \in \RR^2: x = y \geq 0\}
\end{gather*}
is transintegral polyhedral.
\end{enumerate}
\end{theorem}
The set $L$ is a typical tropical line; it is 
the locus where two of the quantities $x,y,0$
are equal and the third is less. That is, $-x,-y$ could be the valuations of
elements $a,b$ of a nonarchimedean ring for which $a+b$ has valuation 0.
\begin{proof}
It is clear that (a) implies (b) and (c). To see that (b) or (c)
implies (a),
first apply Theorem~\ref{T:master} to deduce
that $f$ is locally transintegral polyhedral.
Then apply Lemma~\ref{L:bound1} to each of the four quadrants
(in case (b)) or to the closures of each of the 
three connected components of $\RR^2 \setminus L$
(in case (c))
to deduce that $f$ is transintegral polyhedral.
\end{proof}

Many results in tropical algebraic geometry are analogues of statements
in ordinary algebraic geometry. For example, one might expect 
Theorem~\ref{T:tropical} to be the tropical analogue of a statement
to the effect that a function of two variables is a polynomial if and only if
no matter how we pick one of the variables and a value for that variable, the
result is a polynomial function of the other variable. This statement is correct
under suitable hypotheses, but not in general.
(There is also an analogue for Laurent polynomials, which we leave as 
an exercise.)

\begin{theorem}
Let $F$ be an infinite field. Let $f: F^2 \to F$ be a function for that all
$z_1, z_2 \in F$, the restrictions of $f$ to 
$\{z_1\} \times F$ and $F \times \{z_1\}$ are polynomial functions.
\begin{enumerate}
\item[(a)] If $F$ is uncountable, then $f$ itself must be a polynomial function.
\item[(b)] If $F$ is countable, then $f$ need not be a polynomial
function.
\end{enumerate}
\end{theorem}
We insist that $F$ be infinite so that the evaluation map $F[x] \to F^F$
taking a polynomial to the corresponding function on $F$ is injective.
By contrast, if $F$ is finite, then any function from $F$ to itself
can be expressed as a polynomial function in infinitely many ways.
\begin{proof}
Suppose that for any positive integer $n$, we can find
$z_{1,0}, \dots, z_{0,n} \in F$ such that
the functions $f(z_{1,i}, \cdot): F \to F$ for $i=0,\dots,n$
are linearly independent over $F$.
Then the function $g: F \to F$ defined by
\[
g(\cdot) = \det \begin{pmatrix}
1 & \cdots & 1 \\
z_{1,0} & \cdots & z_{1,n} \\
\vdots & & \vdots \\
z_{1,0}^{n-1} & \cdots & z_{1,n}^{n-1} \\
f(z_{1,0},\cdot) & \cdots & f(z_{1,n}, \cdot)
\end{pmatrix}
\]
cannot be the zero function. However, it is a polynomial in its argument,
so it has only finitely many roots.

If $z_2 \in F$ is such that $f(\cdot,z_2)$ is a polynomial
of degree at most $n-1$, then $g(z_2) = 0$ because we can write the last
row of the matrix as a linear combination of the others. By the previous
paragraph, there are only finitely many such $z_2$. Since this holds for any
$n$, there can only be countably many $z_2$ such that $f(\cdot,z_2)$ is a
polynomial of any degree. Since $F$ is uncountable, this gives a contradiction.

We conclude that the polynomial functions $f(z_1,\cdot)$ span
a finite dimensional vector space over $F$. In particular, they all
represent polynomials of degree bounded by some nonnegative integer $n$.
Choose $z_{2,0},\dots,z_{2,n} \in F$ distinct; for any $z_1,
z_2 \in F$, we now have
\[
0 = \det \begin{pmatrix}
1 & \cdots & 1 & 1\\
z_{2,0} & \cdots & z_{2,n} & z_2\\
\vdots & & \vdots & \vdots \\
z_{2,0}^{n} & \cdots & z_{2,n}^{n} & z_2^n \\
f(z_1,z_{2,0}) &\cdots & f(z_1, z_{2,n}) & f(z_1,z_2)
\end{pmatrix}.
\]
By expanding in minors along the right column,
we express $f(z_1,z_2)$ as a polynomial in $z_2$ whose coefficients
are themselves polynomials in $z_1$. This proves (a).

To prove (b), choose an ordering $t_1, t_2, \dots$ of the set
\[
T = \{\{z_1\} \times F: z_1 \in F\} \cup \{ F \times \{z_2\}: z_2 \in F\}.
\]
(That is, $T$ is the \emph{set} of horizontal and vertical lines, not their
union.) By induction, we may define 
functions $f_n: t_1 \cup \cdots \cup t_n \to F$ such 
that $f_n$ restricts to a polynomial of degree $i$
on $t_i$ for each $i \in \{1,\dots,n\}$. These combine to give a function
$f: F \times F \to F$ whose restriction to each $t \in T$ is a polynomial,
but these restrictions do not have bounded degree. Hence $f$ cannot itself
be a polynomial.
\end{proof}

One might also like to view the fact that a function $f: \RR^2 \to \RR$
is locally transintegral polyhedral if and only if the same is true of
its restriction to every horizontal line and every vertical line
(Theorem~\ref{T:master}) as the tropical analogue of a statement about
polynomials. For this, we might view a locally transintegral polyhedral 
function as the analogue of
something like a Laurent polynomial, but with infinitely many terms. This
suggests formulating a statement about entire functions, such as the 
following. 

\begin{theorem} \label{T:complex}
Let $f: \CC^2 \to \CC$ be a continuous function such that for all 
$z_1, z_2 \in \CC$, the restrictions of $f$ to
$\{z_1\} \times \CC$ and $\CC \times \{z_2\}$ are entire analytic functions.
Then $f$ is an entire analytic function.
\end{theorem}
\begin{proof}
Let $C_i$ be any circle in the $z_i$-plane containing the origin.
For any $z_i$ in the interior of $C_i$, by
the Cauchy integral formula applied twice, we have
\[
\int_{C_1} \int_{C_2} \frac{f(w_1,w_2)}{(w_1-z_1)(w_2-z_2)} dw_2\,dw_1
= \int_{C_1} \frac{f(w_1,z_2)}{w_1 - z_1}\,dw_1 \\
= f(z_1,z_2).
\]
The left side is infinitely differentiable (the continuity of $f$
makes it valid to differentiate under the integral signs), so $f$ 
must be as well. Hence $f$ is entire analytic.
\end{proof}

\begin{question}
Is there an analogue of Theorem~\ref{T:complex} in which the restrictions
of $f$ are only assumed to be meromorphic?
\end{question}

\section{Application to $p$-adic differential equations}

We have cited the papers \cite{kedlaya-part3}, \cite{kedlaya-xiao} for
instances of theorems of the sort we have been discussing. 
To illustrate how these theorems may be applied in practice,
we recall just enough of the theory of $p$-adic differential equations
to articulate one of these applications.

\begin{defn}
Let $F$ be a field of characteristic zero complete for a nonarchimedean
absolute value $|\cdot|$. (We do not require $F$ to be discretely
valued; for instance, we might take $F$ to be $\CC_p$, a completed algebraic
closure of the field of $p$-adic numbers.)
Let $S \subseteq \RR^n$ be a bounded transrational polyhedral set.
Let $R_F(S)$ be the ring whose elements are formal Laurent series
\[
\sum_{i_1,\dots,i_n \in \ZZ}
c_{i_1,\dots,i_n} t_1^{i_1} \cdots t_n^{i_n}
\qquad
(c_{i_1,\dots,i_n} \in F)
\]
such that for each $r = (r_1,\dots,r_n) \in S$,
\[
\lim_{i_1,\dots,i_n \to \pm \infty}
|c_{i_1,\dots,i_n}| e^{-i_1r_1 - \cdots - i_n r_n} = 0.
\]
(This limit should be interpreted as follows: for any $\epsilon > 0$,
there are only finitely many $n$-tuples $(i_1,\dots,i_n)$ for which
the quantity inside the limit is greater than $\epsilon$.) This ring
can be interpreted as the global sections of the structure sheaf of
a certain nonarchimedean analytic space, namely the subset of the
affine $n$-space with coordinates $t_1,\dots,t_n$ defined by the condition
\[
(-\log |t_1|, \dots, -\log |t_n|) \in S.
\]
\end{defn}

\begin{defn}
Let $\Omega$ be the $R_F(S)$-module freely generated by symbols $dt_1,\dots,
dt_n$. 
For $j=1,\dots,n$, define the formal partial derivative $\frac{\partial}{\partial t_j}: R_F(S) \to R_F(S)$ by the formula
\[
\frac{\partial}{\partial t_j}\left(\sum_{i_1,\dots,i_n \in \ZZ}
c_{i_1,\dots,i_n} t_1^{i_1} \cdots t_n^{i_n} \right)
= \sum_{i_1,\dots,i_n \in \ZZ}
i_j t_j^{-1} c_{i_1,\dots,i_n} t_1^{i_1} \cdots t_n^{i_n}.
\]
Define the formal exterior derivative $d: R_F(S) \to \Omega$ by the formula
\[
d(f) = \sum_{j=1}^n \frac{\partial}{\partial t_j}(f)\,dt_j.
\]
Let $M$ be a finite free $R_F(S)$-module. A \emph{connection}
on $M$ is an additive map $\nabla: M \to M \otimes_{R_F(S)} \Omega$ satisfying
the Leibniz rule: for $f \in R_F(S)$ and $m \in M$,
\[
\nabla(fm) = f \nabla(m) + m \otimes d(f).
\]
Given a connection $\nabla$, define the maps $D_1, \dots, D_n: M \to M$
by the formula
\[
\nabla(m) = D_1(m) \otimes dt_1 + \cdots + D_n(m) \otimes dt_n;
\]
note that $D_j$ satisfies the Leibniz rule using the derivation
$\frac{\partial}{\partial t_j}$ on $R_F(S)$.
We say that $\nabla$ is \emph{integrable} if $D_1,\dots,D_n$ commute with
each other. 
\end{defn}

For instance, if $M = R_F(S)$, then $\nabla = d$ is an integrable
connection. We next give a numerical measure of the failure of a given
integrable connection to have this special form.

\begin{defn}
For $r \in S$, define the norm $|\cdot|_r$ on
$R_F(S)$ by the formula
\[
\left|
\sum_{i_1,\dots,i_n \in \ZZ}
c_{i_1,\dots,i_n} t_1^{i_1} \cdots t_n^{i_n}
\right|_r
= \sup\{|c_{i_1,\dots,i_n}| e^{-i_1r_1 - \cdots - i_n r_n}:
i_1,\dots,i_n \in \ZZ\};
\]
the definition of $R_F(S)$ makes this quantity finite, and in fact
ensures that the supremum is achieved for at least one tuple
$(i_1,\dots,i_n) \in \ZZ^n$. Let $E_r$ be the completion of
$\Frac R_F(S)$ for this norm; then each $\frac{\partial}{\partial t_j}$
extends continuously to a map $\frac{\partial}{\partial t_j}:
E_r \to E_r$, while $d$ extends continuously to
a map $d: E_r \to \Omega \otimes_{R_F(S)} E_r$.
\end{defn}

\begin{defn}
Let $M$ be a finite free $R_F(S)$-module equipped with an integrable
connection $\nabla$. For $r \in S$,
we extend $\nabla$ to a map
$\nabla: M \otimes_{R_F(S)} E_r \to M 
\otimes_{R_F(S)} \Omega \otimes_{R_F(S)} E_r$;
we correspondingly extend $D_1,\dots,D_n$ to maps from
$M \otimes_{R_F(S)} E_r$ to itself.
Let $N$ be a subquotient of $M \otimes_{R_F(S)} E_r$
in the category of $E_r$-modules on which $D_1,\dots,D_n$ act. 
Pick a basis of $N$, and use it to define a supremum norm on $N$ compatible
with the norm $|\cdot|_r$ on $E_r$. Let $|D_j|_{\spect,N}$ denote the
\emph{spectral norm} of $D_j$ on $N$, that is,
\[
|D_j|_{\spect,N} = \limsup_{s \to \infty} |D_j^s|^{1/s}_N,
\]
where $|D_j^s|_N$ denotes the operator norm of $D_j^s$ on $N$ for the chosen
norm. This quantity does not depend on the choice of the norm.
Define the \emph{intrinsic generic radius of convergence} of $N$
to be
\[
\min \left\{ \frac{|D_j|_{\spect,E_r}}{|D_j|_{\spect,N}}: j \in \{1,\dots,n\}
\right\}.
\]
(This indeed has something to do with the radius of convergence of certain
horizontal sections of a certain differential module.
See \cite[Proposition~1.2.14]{kedlaya-xiao} for a bit more explanation,
and \cite[\S 9.7]{kedlaya-course} for much more discussion.)

Let $N_1,\dots,N_h$ be the Jordan-H\"older constituents of 
$M \otimes_{R_F(S)} E_r$, i.e., the successive quotients in a maximal filtration
by $E_r$-submodules preserved by $D_1,\dots,D_n$. Define the multiset of
\emph{subsidiary radii} of $M$ at $r$ to consist of, for $j=1,\dots,h$,
the intrinsic generic radius of convergence of $N_j$ with multiplicity
$\dim_{E_r}(N_j)$.
\end{defn}

\begin{remark}
In the case of residual characteristic zero, the subsidiary radii are
related to the classical notion of \emph{irregularity} for a meromorphic
connection on a complex analytic variety. See \cite{kedlaya-good1} for
discussion of that case. In the case of residual characteristic $p>0$,
the subsidiary radii are related to wild ramification
of maps between varieties over a field of positive characteristic;
see \cite{kedlaya-swan1}.
\end{remark}

We can now state part of \cite[Theorem~3.3.9]{kedlaya-xiao},
which governs the variation of the subsidiary radii,
and then explain how a polyhedrality theorem intervenes in the proof.

\begin{theorem}
Let $M$ be a finite free $R_F(S)$-module of rank $m$,
equipped with an integrable
connection $\nabla$. For $r \in S$, let
$f_1(M,r) \geq \cdots \geq f_m(M,r)$ be the nonnegative real numbers
such that the subsidiary radii of $M$ at $r$ are equal to
$e^{-f_1(M,r)}, \dots, e^{-f_m(M,r)}$. Define
$F_i(M,r) = f_1(M,r) + \cdots + f_i(M,r)$.
Then for $i=1,\dots,m$, the function $r \mapsto m! F_i(M,r)$ 
is transintegral polyhedral; for $i=m$, the function $r \mapsto F_m(M,r)$
is also transintegral polyhedral.
\end{theorem}
It turns out to be difficult to analyze the $F_i(M,r)$ on all of $S$
directly. Instead, the proof goes by checking transintegral polyhedrality
on each transrational line, then invoking 
Theorem~\ref{T:kedlaya-xiao}. The point is that by a change of coordinates
(replacing $t_1,\dots,t_n$ by certain monomials) we can reduce the study
of any transrational line to the study of a line parallel to one of the
coordinate axes, say the first one. In that case, there is no harm in burying
$t_2,\dots,t_n$ in the base field (i.e., replacing $F$ by the completion
of the rational function field $F(t_2,\dots,t_n)$ for an appropriate 
Gauss norm) as long as we keep track of the derivations with respect to
$t_2,\dots,t_n$. Although the remaining calculation is still quite intricate,
it would have been immeasurably more so without the reduction to the
one-dimensional case.

\section{Other questions}

We conclude by mentioning some questions for which our ignorance
of the answer remains somewhat frustrating. This should illustrate the
extent to which the restriction to integral slopes is crucial for the
arguments of this paper.

\begin{question}
Let $S$ be an open convex subset of $\RR^n$.
Let $f: S \to \RR$ be a function whose restriction to each line parallel
to one of the coordinate axes is convex. Is $f$ necessarily convex?
\end{question}

\begin{question}
Let $S$ be an open convex subset of $\RR^n$.
Let $f: S \to \RR$ be a function whose restriction to each line parallel
to one of the coordinate axes is locally polyhedral. Is $f$ necessarily 
locally polyhedral?
\end{question}

\begin{question}
Let $f: \RR^2 \to \RR$ be a function whose restriction to each
line is polyhedral.
Is $f$ necessarily polyhedral?
\end{question}

\section*{Acknowledgments}
Thanks to David Speyer for helpful discussions.
The first author was supported by NSF CAREER grant DMS-0545904
and the NEC Research Support Fund. The second author was supported
by the Paul E. Gray Fund of MIT's Undergraduate Research Opportunities
Program (UROP).


\begin{thebibliography}{9}

\bibitem{bourbaki}
N. Bourbaki, \textit{Fonctions d'une variable r\'eelle}, Hermann,
Paris, 1958.

\bibitem{kedlaya-swan1}
K.S. Kedlaya, Swan conductors for $p$-adic differential modules, I: 
A local construction, \textit{Alg. and Num. Theory}
\textbf{1} (2007), 269--300. 

\bibitem{kedlaya-part3}
K.S. Kedlaya, Semistable reduction for overconvergent F-isocrystals, III: 
local semistable reduction at monomial valuations, 
arXiv:math/0609645v3 (2008); to appear in 
\textit{Compos. Math.}

\bibitem{kedlaya-course}
K.S. Kedlaya, $p$-adic differential equations (version of 31 Oct 2008),
preprint available at \texttt{http://math.mit.edu/\~{}kedlaya/papers/}.

\bibitem{kedlaya-good1}
K.S. Kedlaya, Good formal structures for flat meromorphic connections, I:
Surfaces, arXiv:0811.0190v1 (2008).

\bibitem{kedlaya-xiao}
K.S. Kedlaya and L. Xiao, 
Differential modules on $p$-adic polyannuli, 
arXiv:0804.1495v3 (2008); to appear in
\textit{J. Institut Math. Jussieu}.

\bibitem{rock}
R.T. Rockafellar, \textit{Convex analysis}, Princeton Univ. Press,
Princeton, 1970.

\bibitem{young}
D.F. Young, When does unique local support ensure convexity?,
\textit{Trans. Amer. Math. Soc.} \textbf{347} (1995), 1323--1329.

\end{thebibliography}
\end{document}